\documentclass[oneside,american,english,11pt]{amsart}
\usepackage[T1]{fontenc}
\usepackage{babel}
\usepackage{amstext}
\usepackage{amsthm}
\usepackage{amssymb}
\usepackage[unicode=true,pdfusetitle,
 bookmarks=true,bookmarksnumbered=false,bookmarksopen=false,
 breaklinks=false,pdfborder={0 0 1},backref=false,colorlinks=false]
 {hyperref}

\makeatletter
\numberwithin{equation}{section}
\numberwithin{figure}{section}
\theoremstyle{plain}
\newtheorem{thm}{\protect\theoremname}[section]
\theoremstyle{definition}
\newtheorem{defn}[thm]{\protect\definitionname}
\theoremstyle{plain}
\newtheorem{prop}[thm]{\protect\propositionname}
\theoremstyle{plain}
\newtheorem{cor}[thm]{\protect\corollaryname}
\theoremstyle{plain}
\newtheorem{lem}[thm]{\protect\lemmaname}

\usepackage{ae,aecompl} 

\usepackage{fullpage} 

\makeatother

\addto\captionsamerican{\renewcommand{\corollaryname}{Corollary}}
\addto\captionsamerican{\renewcommand{\definitionname}{Definition}}
\addto\captionsamerican{\renewcommand{\lemmaname}{Lemma}}
\addto\captionsamerican{\renewcommand{\propositionname}{Proposition}}
\addto\captionsamerican{\renewcommand{\theoremname}{Theorem}}
\addto\captionsenglish{\renewcommand{\corollaryname}{Corollary}}
\addto\captionsenglish{\renewcommand{\definitionname}{Definition}}
\addto\captionsenglish{\renewcommand{\lemmaname}{Lemma}}
\addto\captionsenglish{\renewcommand{\propositionname}{Proposition}}
\addto\captionsenglish{\renewcommand{\theoremname}{Theorem}}
\providecommand{\corollaryname}{Corollary}
\providecommand{\definitionname}{Definition}
\providecommand{\lemmaname}{Lemma}
\providecommand{\propositionname}{Proposition}
\providecommand{\theoremname}{Theorem}

\begin{document}
\selectlanguage{american}%
\global\long\def\epsilon{\varepsilon}%

\global\long\def\phi{\varphi}%

\global\long\def\RR{\mathbb{R}}%

\global\long\def\SS{\mathbb{S}}%

\global\long\def\oo{\mathbf{1}}%

\global\long\def\dd{\mathrm{d}}%

\global\long\def\div{\operatorname{div}}%

\global\long\def\TV{\operatorname{TV}}%

\global\long\def\per{\operatorname{Per}}%

\global\long\def\cvx{\operatorname{Cvx}}%

\global\long\def\lc{\operatorname{LC}}%

\global\long\def\dom{\operatorname{dom}}%

\global\long\def\HH{\mathcal{H}}%

\title{\selectlanguage{english}%
The anisotropic total variation and surface area measures}
\author{Liran Rotem}
\address{Department of Mathematics, Technion - Israel Institute of Technology,
Israel}
\email{lrotem@technion.edu}
\thanks{The author is partially supported by ISF grant 1468/19 and BSF grant
2016050.}
\begin{abstract}
We prove a a formula for the first variation of the integral of a
log-concave function, which allows us to define the surface area measure
of such a function. The formula holds in complete generality with
no regularity assumptions, and is intimately related to the notion
of anisotropic total variation and to anisotropic coarea formulas.
This improves previous partial results by Colesanti and Fragal\`{a},
by Cordero-Erausquin and Klartag and by the author. 
\end{abstract}

\maketitle

\section{\label{sec:introduction}Introduction}

One of the basic constructions in convex geometry is the surface area
measure of a convex body. To recall the definition, let $K\subseteq\RR^{n}$
be a convex body, i.e a compact convex set with non-empty interior.
Then the Gauss map $n_{K}:\partial K\to\SS^{n-1}$ exists $\HH^{n-1}$-almost
everywhere, where $\HH^{n-1}$ denotes the $(n-1)$-dimensional Hausdorff
measure and $\SS^{n-1}=\left\{ x\in\RR^{n}:\ \left|x\right|=1\right\} $
denotes the unit sphere. We then define the surface area measure $S_{K}$
as the push-forward $S_{K}=\left(n_{K}\right)_{\sharp}\left(\left.\HH^{n-1}\right|_{\partial K}\right)$.
More explicitly, for every measurable function $\phi:\SS^{n-1}\to\RR$
we have 
\[
\int_{\SS^{n-1}}\phi\dd S_{K}=\int_{\partial K}\left(\phi\circ n_{K}\right)\dd\HH^{n-1}.
\]

The surface area measure can be equivalently defined using the first
variation of volume. For convex bodies $K$ and $L$, let $K+L$ denotes
the usual Minkowski sum, and let $\left|K\right|$ denote the (Lebesgue)
volume of $K$. Then we have 
\begin{equation}
\lim_{t\to0^{+}}\frac{\left|K+tL\right|-\left|K\right|}{t}=\int_{\SS^{n-1}}h_{L}\dd S_{K},\label{eq:surface-bodies}
\end{equation}
 where $h_{L}:\SS^{n-1}\to\RR$ is the support function of $L$ which
is defined by $h_{L}(\theta)=\max_{x\in L}\left\langle x,\theta\right\rangle $.
For a proof of this fact the reader may consult a standard reference
book in convex geometry such as \cite{Schneider2013} or \cite{Hug2020}. 

In this paper we will be interested in functional extensions of the
formula \eqref{eq:surface-bodies}. Recall that a function $f:\RR^{n}\to[0,\infty)$
is called log-concave if 
\[
f\left((1-\lambda)x+\lambda y\right)\ge f(x)^{1-\lambda}f(y)^{\lambda}
\]
 for all $x,y\in\RR^{n}$ and $0\le\lambda\le1$. We denote by $\lc_{n}$
the class of all upper semi-continuous log-concave functions. Note
that the class of convex bodies in $\RR^{n}$ embeds naturally into
$\lc_{n}$ using the map 
\[
K\hookrightarrow\oo_{K}(x)=\begin{cases}
1 & x\in K\\
0 & \text{otherwise. }
\end{cases}
\]
We would like to consider log-concave functions as ``generalized
convex bodies'', and extend the notion of the surface area measure
to this setting. To achieve this goal we first need to recall the
standard operations on log-concave functions: The sum of two log-concave
functions is given by the sup-convolution, i.e. 
\[
\left(f\star g\right)(x)=\sup_{y\in\RR^{n}}\left(f(y)g(x-y)\right).
\]
 The associated dilation operation is given by $\left(\lambda\cdot f\right)(x)=f\left(\frac{x}{\lambda}\right)^{\lambda}$
\textendash{} note that we have for example $f\star f=2\cdot f$.
Finally, the ``volume'' of $f$ will be given by the Lebesgue integral
$\int f$. Using these constructions we may define:
\begin{defn}
Fix $f,g\in\lc_{n}$. The first variation of the integral of $f$
in the direction of $g$ is given by 
\begin{equation}
\delta(f,g)=\lim_{t\to0^{+}}\frac{\int f\star\left(t\cdot g\right)-\int f}{t}.\label{eq:variation-def}
\end{equation}
\end{defn}

The study of log-concave functions as geometric objects has become
a major idea in convex geometry, with useful applications even if
eventually one is only interested in convex bodies. Due to the very
large number of papers in this direction we will not survey all of
them, but only mention the ones that directly deal with the first
variation $\delta(f,g)$. In the case when $f=e^{-\frac{\left|x\right|^{2}}{2}}$
is a Gaussian, $\delta(f,g)$ was studied under the name ``the mean
width of $g$'' by Klartag and Milman (\cite{Klartag2005}) in one
of the papers that began the geometric study of log-concave functions.
This mean width was further studied in \cite{Rotem2012}. In particular
it was proved there that in this case we have 
\[
\delta(f,g)=\int_{\RR^{n}}h_{g}(x)e^{-\left|x\right|^{2}/2}\dd x.
\]
 Here $h_{g}:\RR^{n}\to\RR$ is the support function of $g$, defined
by $h_{g}=(-\log g)^{\ast}$ where 
\[
\phi^{\ast}(x)=\sup_{y\in\RR^{n}}\left(\left\langle x,y\right\rangle -\phi(y)\right)
\]
 is the Legendre transform. 

The case of a general function $f$ was studied by Colesanti and Fragal\`{a}
in \cite{Colesanti2013}. In particular they showed that the limit
in \eqref{eq:variation-def} always exists when $0<\int f<\infty$,
though it may be equal to $+\infty$. To further explain their results
we will need some important definitions:
\begin{defn}
Fix a log-concave function $f:\RR^{n}\to\RR$ with $0<\int f<\infty$,
and write $f=e^{-\phi}$ for a convex function $\phi:\RR^{n}\to(-\infty,\infty]$.
Then: 
\begin{enumerate}
\item The measure $\mu_{f}$ is a measure on $\RR^{n}$ defined as the push-forward
$\mu_{f}=\left(\nabla\phi\right)_{\sharp}\left(f\dd x\right)$. 
\item The measure $\nu_{f}$ is a measure on the sphere $\SS^{n-1}$, defined
as the push-forward $\nu_{f}=\left(n_{K_{f}}\right)_{\sharp}\left(f\dd\!\left.\HH^{n-1}\right|_{\partial K_{f}}\right)$.
Here $K_{f}$ is a shorthand notation for the support of $f$, i.e.
$K_{f}=\left\{ x\in\RR^{n}:\ f(x)>0\right\} $, and $n_{K_{f}}$ denotes
the Gauss map $n_{K_{f}}:\partial K_{f}\to\SS^{n-1}$. 
\end{enumerate}
\end{defn}

For example, for $f(x)=e^{-\left|x\right|^{2}/2}$ we have $\mu_{f}=e^{-\left|x\right|^{2}/2}\dd x$
and $\nu_{f}\equiv0$, as $\partial K_{f}=\partial\RR^{n}=\emptyset$.
For a convex body $K$ we have $\mu_{\oo_{K}}=\left|K\right|\delta_{0}$
and $\nu_{\oo_{K}}=S_{K}$, the usual surface area measure. 

It will be important for us to observe that no regularity is required
for the definitions of $\mu_{f}$ and $\nu_{f}$. Indeed, as $\phi=-\log f$
is a convex function it is differentiable Lebesgue-almost-everywhere
on the set $\left\{ x:\ \phi(x)<\infty\right\} =K_{f}$. Therefore
the push-forward $\left(\nabla\phi\right)_{\sharp}\left(f\dd x\right)$
is well-defined. Similarly since $K_{f}$ is a closed convex set its
boundary $\partial K_{f}$ is a Lipschitz manifold, so in particular
the Gauss map $n_{K_{f}}$ is defined $\HH^{n-1}$-almost-everywhere
and the push-forward is again well-defined. 

While regularity is not needed for the definitions of $\mu_{f}$ and
$\nu_{f}$, it was definitely needed for the representation theorem
of \cite{Colesanti2013}:
\begin{thm}[Colesanti\textendash Fragal\`{a}]
\label{thm:colesanti-fragala}Fix $f,g\in\lc_{n}$, and assume that: 
\begin{enumerate}
\item The supports $K_{f}$, $K_{g}$ are $C^{2}$ smooth convex bodies
with everywhere positive Gauss curvature. 
\item The functions $\psi=-\log f$ and $\phi=-\log g$ are continuous in
$K_{f}$ and $K_{g}$ respectively, $C^{2}$ smooth in the interior
of these sets, and have strictly positive-definite Hessians.
\item We have $\lim_{x\to\partial K_{f}}\left|\nabla\psi(x)\right|=\lim_{x\to\partial K_{g}}\left|\nabla\phi(x)\right|=\infty$. 
\item The difference $h_{f}-c\cdot h_{g}$ is convex for small enough $c>0$. 
\end{enumerate}
Then 
\[
\delta(f,g)=\int_{\RR^{n}}h_{g}\dd\mu_{f}+\int_{\SS^{n-1}}h_{K_{g}}\dd\nu_{f}.
\]
 
\end{thm}

Based on Theorem \ref{thm:colesanti-fragala} we can make the following
definition:
\begin{defn}
Given $f\in\lc_{n}$ with $0<\int f<\infty$, we call the pair $\left(\mu_{f},\nu_{f}\right)$
the surface area measure\textbf{s} of the function $f$. 
\end{defn}

We emphasize that unlike a convex body, a log-concave function has
two surface area measures: one defined on $\RR^{n}$, and one defined
on $\SS^{n-1}$. 

While the regularity assumptions of Theorem \ref{thm:colesanti-fragala}
are sufficient, it has always been clear that they are not necessary.
For example, we already saw that for the function $f(x)=e^{-\left|x\right|^{2}/2}$
no regularity assumptions on $g$ are needed. In fact, it was proved
in \cite{Rotem2021} that if $0<\int f<\infty$ and $\nu_{f}=0$ then
we have 
\[
\delta(f,g)=\int_{\RR^{n}}h_{g}\dd\mu_{f}
\]
 with no regularity assumptions. Note that since $f$ is log-concave
and upper semi-continuous it is only discontinuous at points $x\in\partial K_{f}$
such that $f(x)\ne0$. Therefore the condition $\nu_{f}=0$ is equivalent
to the statement that $f$ is continuous $\HH^{n-1}$-almost everywhere.
This property was dubbed essential continuity by Cordero-Erausquin
and Klartag (\cite{Cordero-Erausquin2015}). In their paper, the authors
\cite{Cordero-Erausquin2015} studied the moment measure of a convex
function $\phi$, which in our terminology is simply the surface area
measure $\mu_{e^{-\phi}}$. One of the main results of their paper
is a functional analogue of Minkowski's existence theorem: Given a
measure $\mu$ on $\RR^{n}$, they provide a necessary and sufficient
condition for the existence of a function $f\in\lc_{n}$ with $\mu_{f}=\mu$
and $\nu_{f}=0$. They also prove the uniqueness of such an $f$.
We remark that when the functions involved are not necessarily essentially
continuous, but are sufficiently regular in the sense of Theorem \ref{thm:colesanti-fragala},
a similar uniqueness result was previously proved by Colesanti\textendash Fragal\`{a}
in \cite{Colesanti2013}. We will further discuss this issue of uniqueness
in Section \ref{sec:completing-the-proof} after proving our main
theorem, and explain the results of both papers. We also remark that
another proof of the same existence theorem was given by Santambrogio
in \cite{Santambrogio2016}.

The main goal of this paper is to prove the most general form of Theorem
\ref{thm:colesanti-fragala}, which requires no regularity assumptions:
\begin{thm}
\label{thm:main-thm}Fix $f,g\in\lc_{n}$ such that $0<\int f<\infty$.
Then 
\begin{equation}
\delta(f,g)=\int_{\RR^{n}}h_{g}\dd\mu_{f}+\int_{\SS^{n-1}}h_{K_{g}}\dd\nu_{f}.\label{eq:main-formula}
\end{equation}
\end{thm}

While the improvement over previous results is simply the elimination
of the various technical assumptions, we do believe Theorem \ref{thm:main-thm}
is of real value. For example, we will see as a corollary that the
pair of measures $\left(\mu_{f},\nu_{f}\right)$ determines $f$ uniquely,
and for this result it is very useful not to have any technical conditions
for the validity of formula \eqref{eq:main-formula}. Moreover, we
believe our proof sheds some light on the reason this formula holds.
In particular, we will see an interesting connection between Theorem
\ref{thm:main-thm} and the notions of anisotropic total variation
and anisotropic perimeter. The main point will be that when $g=\oo_{L}$
and $f\in\lc_{n}$ is arbitrary, Theorem \ref{thm:main-thm} can be
viewed as an anisotropic version of the coarea formula. 

The rest of this note is dedicated to the proof of the theorem. In
Section \ref{sec:anisotropic-total-variations} we will introduce
the anisotropic coarea formula, and explain why it is in fact equivalent
to Theorem \ref{thm:main-thm} in the case when $g=\oo_{L}$ is the
indicator of a convex body. Then in Section \ref{sec:completing-the-proof}
we will discuss the case of a general function $g$, and conclude
the proof. Some of the ingredients that were used in previous results
(mostly in \cite{Rotem2021}) can be used in the proof of Theorem
\ref{thm:main-thm} with few changes, and in these cases we will either
give an exact reference or briefly sketch the argument. 

Before we start with the main proof let us prove a finiteness property
of the measures $\mu_{f}$ and $\nu_{f}$: 
\begin{prop}
\label{prop:finiteness}Assume $f\in\lc_{n}$ and $0<\int f<\infty$.
Then the measure $\mu_{f}$ is finite with a finite first moment.
The measure $\nu_{f}$ is also finite. 
\end{prop}

\begin{proof}
$\mu_{f}$ is finite by definition, as $\int_{\RR^{n}}\dd\mu_{f}=\int f\dd x$.
Moreover, we have 
\[
\int\left|x\right|\dd\mu_{f}=\int\left|\nabla\phi\right|f\dd x=\int\left|\nabla f\right|\dd x<\infty,
\]
 where the last inequality is part of Lemma 4 of \cite{Cordero-Erausquin2015}. 

Next we show that $\nu_{f}$ is finite. Note that this is not entirely
trivial since $\int\dd\nu_{f}=\int_{\partial K_{f}}f\dd\HH^{n-1}$,
and while $f$ is clearly bounded we can have $\HH^{n-1}\left(\partial K_{f}\right)=\infty$.
We therefore adapt a simple argument of Ball (\cite{Ball1993}). Since
$f$ is log-concave and integrable, there exists constants $A,c>0$
such that $f(x)\le Ae^{-c\left|x\right|}$ (see e.g. Lemma 2.1 of
\cite{Klartag2007}). Note that for all $x\in\RR^{n}$ we have
\[
e^{-c\left|x\right|}=c\int_{0}^{\infty}e^{-ct}\oo_{tB}(x)\dd t,
\]
 where $B=\left\{ x:\ \left|x\right|\le1\right\} $ is the unit ball.
We may therefore compute 
\begin{align*}
\int\dd\nu_{f} & =\int_{\partial K_{f}}f\dd\HH^{n-1}\le A\int_{\partial K_{f}}e^{-c\left|x\right|}\dd\HH^{n-1}=Ac\int_{0}^{\infty}\int_{\partial K_{f}}e^{-ct}\oo_{tB}(x)\dd\HH^{n-1}(x)\dd t\\
 & =Ac\cdot\int_{0}^{\infty}e^{-ct}\HH^{n-1}\left(\partial K_{f}\cap tB\right)\dd t\le Ac\cdot\int_{0}^{\infty}e^{-ct}\HH^{n-1}\left(\partial\left(K_{f}\cap tB\right)\right)\dd t\\
 & \le Ac\cdot\int_{0}^{\infty}e^{-ct}\HH^{n-1}\left(\partial\left(tB\right)\right)\dd t=Ac\cdot\HH^{n-1}\left(\SS^{n-1}\right)\cdot\int_{0}^{\infty}t^{n-1}e^{-ct}\dd t<\infty,
\end{align*}
which is what we wanted to prove. Note that the last inequality holds
since $K_{f}\cap tB\subseteq tB$ and surface area is monotone for
convex bodies. 
\end{proof}
In particular, it follows that the equality in \eqref{eq:main-formula}
is an equality of finite quantities whenever $g$ is compactly supported,
as in this case $h_{g}\le A\left|x\right|+B$ and $h_{K_{g}}$ is
bounded on $\SS^{n-1}$. When $g$ is not compactly supported it is
possible to have $\delta(f,g)=\infty$, as already mentioned. 

\section{\label{sec:anisotropic-total-variations}Anisotropic total variations}

In order to start our proof, we need the notion of the anisotropic
total variation. First recall the classical (isotropic) total variation:
An integrable function $f:\RR^{n}\to\RR$ is said to have bounded
variation if 
\[
\sup\left\{ \int_{\RR^{n}}f\div\Phi\dd x:\ \begin{array}{l}
\Phi:\RR^{n}\to\RR^{n}\text{ is }C^{1}\text{-smooth, compactly}\\
\text{supported and }\left|\Phi(x)\right|\le1\text{ for all }x\in\RR^{n}
\end{array}\right\} <\infty.
\]
 This supremum is then known as the total variation of $f$, which
we will denote by $\TV(f)$. Moreover, if $f$ is of bounded variation
then there exists a vector-valued measure $Df$ on $\RR^{n}$ such
that 
\[
\int_{\RR^{n}}f\div\Phi\dd x=-\int_{\RR^{n}}\left\langle \Phi,\dd\left(Df\right)\right\rangle ,
\]
 and $\TV(f)=\left|Df\right|\left(\RR^{n}\right)$. Here $\left|Df\right|$
denotes the total variation (in the sense of measures) of $Df$. 

A set $A\subseteq\RR^{n}$ is said to have finite perimeter if $\oo_{A}$
has finite variation, and we define $\per(A)=\TV\left(\oo_{A}\right)$.
Finally, the coarea formula states that if $f$ has bounded variation
then $F_{s}=\left\{ x:\ f(x)\ge s\right\} $ has finite perimeter
for almost every $s$, and $\TV(f)=\int_{-\infty}^{\infty}\per(F_{s})\dd s$.
All of these facts are standard \textendash{} see e.g. Chapter 5 of
\cite{Evans1992} for proofs. 

It is less well known that the role of Euclidean norm in the definition
of $\TV(f)$ is not essential. Fix a convex body $L\subseteq\RR^{n}$
and assume that $0$ belongs to the interior of $L$. Then the (non-symmetric)
norm
\[
\left\Vert x\right\Vert _{L}=\inf\left\{ \lambda>0:\ \frac{x}{\lambda}\in L\right\} 
\]
 is equivalent to the Euclidean norm. We then define:
\begin{defn}
\begin{enumerate}
\item Let $f:\RR^{n}\to\RR^{n}$ be an integrable function. Then the $L$-total
variation of $f$ is given by 
\[
\TV_{L}(f)=\sup\left\{ \int_{\RR^{n}}f\div\Phi\dd x:\ \begin{array}{l}
\Phi:\RR^{n}\to\RR^{n}\text{ is }C^{1}\text{-smooth, compactly}\\
\text{supported and }\left\Vert \Phi(x)\right\Vert _{L}\le1\text{ for all }x\in\RR^{n}
\end{array}\right\} .
\]
\item Let $A\subseteq\RR^{n}$ be a measurable set. The $L$-perimeter of
$A$ is defined by $\per_{L}(A)=\TV_{L}(\oo_{A}).$ 
\end{enumerate}
Since $\left\Vert \cdot\right\Vert _{L}$ and $\left|\cdot\right|$
are equivalent the notion of ``bounded variation'' does not depend
on $L$, and $\TV_{L}(f)<\infty$ if and only if $\TV(f)<\infty$.
Of course, the variation itself does depend on $L$.

The theory of anisotropic total variations is analogous to the standard
theory. We now cite two results that we will require. We were only
able to find as reference the technical report \cite{Grasmair2010},
where these results are proven by Grasmair, but the results can be
proved in the same way as the classical proofs that appear e.g. in
\cite{Evans1992}:
\end{defn}

\begin{prop}
\label{prop:TV-def}Let $f:\RR^{n}\to\RR$ be of bounded variation.
Write the vector valued measure $Df$ as $Df=\sigma\mu$ where $\mu$
is a positive measure and $\sigma:\RR^{n}\to\RR^{n}$ satisfies $h_{L}(\sigma(x))=1$
for all $x\in\RR^{n}$. Then $\TV_{L}(f)=\mu\left(\RR^{n}\right).$ 
\end{prop}

Note that when $L$ is the Euclidean ball the measure $\mu$ is exactly
$\left|Df\right|$, the usual total variation of $Df$. In the general
case $\mu$ can be considered as ``total variation of $Df$ with
respect to $L$''. Also note that $\TV_{L}(f)$ was defined using
the norm $\left\Vert \cdot\right\Vert _{L}$, but in Proposition \ref{prop:TV-def}
the norm that appears in the dual norm $h_{L}$. 

We will also need the anisotropic coarea formula:
\begin{thm}
\label{thm:coarea}Fix an integrable $f:\RR^{n}\to\RR$, and denote
its level sets by $F_{s}=\left\{ x:\ f(x)\ge s\right\} $. Then
\[
\TV_{L}(f)=\int_{-\infty}^{\infty}\per_{L}(F_{s})\dd s.
\]
\end{thm}

Our goal for this section is to prove Theorem \ref{thm:main-thm}
when $g=\oo_{L}$ , the indicator of a convex body. We will do so
by proving that in this case Theorem \ref{thm:main-thm} is essentially
equivalent to Theorem \ref{thm:coarea}. We begin with finding an
alternative formula for $\delta(f,g)$ in this case. Recall that if
$K,L\subseteq\RR^{n}$ are convex bodies then the volume $\left|K+tL\right|$
for $t\ge0$ can be written as a polynomial,
\begin{equation}
\left|K+tL\right|=\sum_{k=0}^{n}\binom{n}{k}W_{k}(K,L)t^{k}.\label{eq:steiner}
\end{equation}
The non-negative coefficients $W_{k}(K,L)$ are known in our normalization
as the relative quermassintegrals of $K$ with respect to $L$. Formula
\eqref{eq:steiner} is a special case of the celebrated Minkowski
theorem, and for the proof and basic properties of the relative quermassintegrals
we refer the reader again to \cite{Schneider2013} or \cite{Hug2020}.
For now we just note that $W_{0}(K,L)=\left|K\right|$. 

We now prove: 
\begin{prop}
\label{prop:quer-repr}Fix $f\in\lc_{n}$ with $0<\int f<\infty$
and fix a compact convex body $L\subseteq\RR^{n}$. For every $s>0$
we write $F_{s}=\left\{ x\in\RR^{n}:\ f(x)\ge s\right\} $. Then 
\[
\delta(f,\oo_{L})=n\int_{0}^{\infty}W_{1}(F_{s},L)\dd s.
\]
\end{prop}

This result essentially appears in \cite{Bobkov2014}, at least in
the case when $L$ is the unit ball. Nonetheless we present its short
proof:
\begin{proof}
For brevity we define $f_{t}=f\star\left(t\cdot\oo_{L}\right)=f\star\oo_{tL}$
and $F_{s}^{(t)}=\left\{ x\in\RR^{n}:\ f_{t}(x)\ge s\right\} $. It
is immediate from the definition of $f_{t}$ that $F_{s}^{(t)}=F_{s}+tL$.
By layer cake decomposition we have 
\[
\delta(f,g)=\lim_{t\to0^{+}}\frac{\int f_{t}-\int f}{t}=\lim_{t\to0^{+}}\frac{\int_{0}^{\infty}\left|F_{s}^{(t)}\right|\dd s-\int_{0}^{\infty}\left|F_{s}\right|\dd s}{t}=\lim_{t\to0^{+}}\int_{0}^{\infty}\frac{\left|F_{s}+tL\right|-\left|F_{s}\right|}{t}\dd s.
\]
Using \eqref{eq:steiner} we see that for $0<t<1$ we have
\[
0\le\frac{\left|F_{s}+tL\right|-\left|F_{s}\right|}{t}=\sum_{k=1}^{n}\binom{n}{k}W_{k}(F_{s},L)t^{k-1}\le\sum_{k=1}^{n}\binom{n}{k}W_{k}(F_{s},L)=\left|F_{s}+L\right|-\left|F_{s}\right|.
\]
Since $\int_{0}^{\infty}\left(\left|F_{s}+L\right|-\left|F_{s}\right|\right)\dd s=\int f_{1}-\int f<\infty$,
we can use the dominated convergence theorem to conclude that 
\[
\delta(f,g)=\int_{0}^{\infty}\left(\lim_{t\to0^{+}}\frac{\left|F_{s}+tL\right|-\left|F_{s}\right|}{t}\right)\dd s=n\int_{0}^{\infty}W_{1}(F_{s},L)\dd s.
\]
\end{proof}
We will also need one more identity, which was proved in \cite{Rotem2021}
as part of Theorem 3.2:
\begin{prop}
\label{prop:variation-measure}Let $\Phi:\RR^{n}\to\RR^{n}$ denote
a $C^{1}$-smooth compactly supported vector field. Fix $f\in\lc_{n}$
with $0<\int f<\infty$. Then 
\[
\int_{\RR^{n}}f\div\Phi\dd x=-\int_{\RR^{n}}\left\langle \nabla f,\Phi\right\rangle \dd x+\int_{\partial K_{f}}f\left\langle \Phi,n_{K_{f}}\right\rangle \dd\HH^{n-1}.
\]
\end{prop}

Armed with these tools, we can finally prove:
\begin{prop}
\label{prop:result-bodies}Fix $f\in\lc_{n}$ with $0<\int f<\infty$,
and set $g=c\cdot\oo_{L}$ for a compact convex set $L\subseteq\RR^{n}$
and $c>0$. Then
\[
\delta(f,g)=\int_{\RR^{n}}h_{g}\dd\mu_{f}+\int_{\SS^{n-1}}h_{K_{g}}\dd\nu_{f}.
\]
 
\end{prop}

\begin{proof}
First note that if the result holds for a function $g$, it also holds
for $\widetilde{g}(x)=e^{c}\cdot g(x)$ for all $c\in\RR$. Indeed,
it is easy to check by the chain rule that $\delta(f,\widetilde{g})=\delta(f,g)+c\int f$
(see Proposition 3.4 in \cite{Rotem2021}). Since $h_{\widetilde{g}}=h_{g}+c$
and $h_{K_{g}}=h_{K_{\widetilde{g}}}$, we have 
\begin{align*}
\delta(f,\widetilde{g}) & =\delta(f,g)+c\int f=\int_{\RR^{n}}h_{g}\dd\mu_{f}+\int_{\SS^{n-1}}h_{K_{g}}\dd\nu_{f}+c\int f\\
 & =\int_{\RR^{n}}(h_{g}+c)\dd\mu_{f}+\int_{\SS^{n-1}}h_{K_{g}}\dd\nu_{f}=\delta(f,g)=\int_{\RR^{n}}h_{\widetilde{g}}\dd\mu_{f}+\int_{\SS^{n-1}}h_{K_{\widetilde{g}}}\dd\nu_{f}
\end{align*}
 as claimed. Therefore from now on we can (and will) assume that $g=\oo_{L}$. 

Next, assume that $0$ belongs to the interior of $L$ so the theory
of anisotropic total variations applies. Proposition \ref{prop:variation-measure}
immediately implies that
\[
\dd\left(Df\right)=-\nabla f\dd x+fn_{\partial K_{f}}\left.\dd\HH^{n-1}\right|_{\partial K_{f}}.
\]
 Therefore the measure $\mu$ from Proposition \ref{prop:TV-def}
is $\dd\mu=h_{L}\left(-\nabla f\right)\dd x+fh_{L}\left(n_{\partial K_{f}}\right)\left.\dd\HH^{n-1}\right|_{\partial K_{f}}$.
By the same proposition we then have 
\begin{align}
TV_{L}(f) & =\int\dd\mu=\int_{\RR^{n}}h_{L}(-\nabla f)\dd x+\int_{\partial K_{f}}fh_{L}\left(n_{\partial K_{f}}\right)\dd\HH^{n-1}\label{eq:tv-repr}\\
 & =\int_{\RR^{n}}h_{L}(\nabla\phi)f\dd x+\int_{\partial K_{f}}h_{L}\left(n_{\partial K_{f}}\right)f\dd\HH^{n-1}\nonumber \\
 & =\int h_{L}\dd\mu_{f}+\int h_{L}\dd\nu_{f}=\int h_{g}\dd\mu_{f}+\int h_{K_{g}}\dd\nu_{f},\nonumber 
\end{align}
 where of course $\phi=-\log f$. 

As \eqref{eq:tv-repr} holds for all $f\in\lc_{n}$ with $0<\int f<\infty$
we can in particular apply this formula to the indicator $\oo_{F}$
of a convex body $F$. We then obtain 
\begin{equation}
\per_{L}(F)=\TV_{L}\left(\oo_{F}\right)=\int h_{L}\dd\left(\left|F\right|\delta_{0}\right)+\int h_{L}\dd S_{F}=\int h_{L}\dd S_{F}=nW_{1}(F,L),\label{eq:per-quer}
\end{equation}
 where the last equality is a standard (and follows immediately from
\eqref{eq:surface-bodies} and \eqref{eq:steiner}). 

Therefore, using in order Proposition \ref{prop:quer-repr}, equation
\eqref{eq:per-quer}, Theorem \ref{thm:coarea} and equation \eqref{eq:tv-repr}
we have 
\[
\delta(f,g)=n\int_{0}^{\infty}W_{1}(F_{s},L)\dd s=\int_{0}^{\infty}\per_{L}\left(F_{s}\right)\dd s=\TV_{L}(f)=\int h_{g}\dd\mu_{f}+\int h_{K_{g}}\dd\nu_{f}.
\]
This concludes the proof in the case $0\in\text{int}(L)$. 

For the general case, fix a large Euclidean ball $B$ centered at
the origin such that $B+L$ contains the origin in its interior. From
Proposition \ref{prop:quer-repr} and standard properties of quermassintegrals
it follows that $\delta(f,\oo_{L})$ is linear in $L$ with respect
to the Minkowski addition. Therefore 
\begin{align*}
\delta(f,\oo_{L}) & =\delta(f,\oo_{L+B})-\delta(f,\oo_{B})\\
 & =\left(\int h_{L+B}\dd\mu_{f}+\int h_{L+B}\dd\nu_{f}\right)-\left(\int h_{B}\dd\mu_{f}+\int h_{B}\dd\nu_{f}\right)\\
 & =\int h_{L}\dd\mu_{f}+\int h_{L}\dd\nu_{f}
\end{align*}
 and the proof is complete. 
\end{proof}
Note that as a corollary we obtain the following result:
\begin{cor}
For $f\in\lc_{n}$ with $0<\int f<\infty$ the sum $\mu_{f}+\nu_{f}$
is centered, i.e. for all $v\in\RR^{n}$ we have 
\[
\int_{\RR^{n}}\left\langle x,v\right\rangle \dd\mu_{f}+\int_{\SS^{n-1}}\left\langle x,v\right\rangle \dd\nu_{f}=0.
\]
\end{cor}

\begin{proof}
Simply take $g=\oo_{\left\{ v\right\} }$ in Proposition \ref{prop:result-bodies},
and note that $\delta(f,g)=0$. 
\end{proof}
The fact that $\mu_{f}$ is centered when $\nu_{f}=0$ was observed
already in \cite{Cordero-Erausquin2015}. 

\section{\label{sec:completing-the-proof}Completing the proof}

In this section we finish the proof of Theorem \ref{thm:main-thm}.
We start with the case of compactly supported $g$. The following
lemma from \cite{Rotem2021} will be crucial:
\begin{lem}
\label{lem:der-formula}Fix $f,g\in\lc_{n}$ such that $0<\int f<\infty$
and $g$ is compactly supported. Then for (Lebesgue) almost every
$x\in\RR^{n}$ we have 
\[
\lim_{t\to0^{+}}\frac{\left(f\star\left(t\cdot g\right)\right)(x)-f(x)}{t}=h_{g}\left(\nabla\phi(x)\right)f(x).
\]
 Here $\phi=-\log f$, and the right hand side is interpreted at $0$
whenever $f(x)=0$. 
\end{lem}

This lemma is proved in \cite{Rotem2021} as part of the proof of
Lemma 3.7 (The condition $h_{g}(y)\le m\left|y\right|+c$ in the statement
of that lemma is exactly equivalent to $g$ being compactly supported).
If all functions involved are sufficiently regular Lemma \ref{lem:der-formula}
follows from the standard formula for the first variation of the Legendre
transform (see more information in \cite{Rotem2021}). To prove this
result with no regularity assumptions does take a bit of work which
we will not reproduce here. 

We will now prove:
\begin{prop}
\label{prop:main-linear-growth}Fix $f,g\in\lc_{n}$ such that $0<\int f<\infty$
and $g$ is compactly supported. Then 
\[
\delta(f,g)=\int_{\RR^{n}}h_{g}\dd\mu_{f}+\int_{\SS^{n-1}}h_{K_{g}}\dd\nu_{f}.
\]
\end{prop}

\begin{proof}
To simplify our notation let us define $f_{t}=f\star\left(t\cdot g\right)$.
We also choose $A>0$ such that $0\le g(x)\le A$ for all $x\in\RR^{n}$,
and we define $\widetilde{g}=A\cdot\oo_{K_{g}}$ and $\widetilde{f}_{t}=f\star\left(t\cdot\widetilde{g}\right)$.
Note that $g\le\widetilde{g}$, so $f_{t}\le\widetilde{f}_{t}$ for
all $t>0$. Also note that 
\[
\lim_{t\to0^{+}}\frac{\widetilde{f}_{t}-f_{t}}{t}=\lim_{t\to0^{+}}\left(\frac{\widetilde{f}_{t}-f}{t}-\frac{f_{t}-f}{t}\right)=h_{\widetilde{g}}\left(\nabla\phi\right)f-h_{g}\left(\nabla\phi\right)f
\]
almost everywhere, where we used Lemma \ref{lem:der-formula} twice.
We may therefore apply Fatou's lemma and deduce that 
\begin{align*}
\liminf_{t\to0^{+}}\left(\frac{\int\widetilde{f}_{t}-\int f}{t}-\frac{\int f_{t}-\int f}{t}\right) & =\liminf_{t\to0^{+}}\int\frac{\widetilde{f}_{t}-f_{t}}{t}\dd x\ge\int\left(h_{\widetilde{g}}\left(\nabla\phi\right)f-h_{g}\left(\nabla\phi\right)f\right)\dd x\\
 & =\int\left(h_{\widetilde{g}}-h_{g}\right)\dd\mu_{f}.
\end{align*}
 However, by Proposition \ref{prop:result-bodies} we know that 
\[
\lim_{t\to0^{+}}\frac{\int\widetilde{f}_{t}-\int f}{t}=\delta(f,\widetilde{g})=\int h_{\widetilde{g}}\dd\mu_{f}+\int h_{K_{g}}\dd\nu_{f},
\]
 where we used the fact that $K_{\tilde{g}}=K_{g}$. Combining the
last two formulas we see that 
\[
\int h_{\widetilde{g}}\dd\mu_{f}+\int h_{K_{g}}\dd\nu_{f}-\limsup_{t\to0^{+}}\left(\frac{\int f_{t}-\int f}{t}\right)\ge\int\left(h_{\widetilde{g}}-h_{g}\right)\dd\mu_{f},
\]
 so 
\begin{equation}
\limsup_{t\to0^{+}}\frac{\int f_{t}-\int f}{t}\le\int h_{g}\dd\mu_{f}+\int h_{K_{g}}\dd\nu_{f}.\label{eq:main-upper-bound}
\end{equation}
 Note that we were allowed to cancel $\int h_{\widetilde{g}}\dd\mu_{f}$
from both sides since this expression is finite by Proposition \ref{prop:finiteness}. 

The proof of the opposite inequality is similar. Fix $m\in\mathbb{N}$
and consider $K_{m}=\left\{ x\in K_{g}:\ g(x)\ge\frac{1}{m}\right\} $.
This time we define $\widetilde{g}=\frac{1}{m}\oo_{K_{m}}$ and $\widetilde{f}_{t}=f\star\left(t\cdot\widetilde{g}\right)$,
and we have the opposite inequality $f_{t}\ge\widetilde{f}_{t}$.
Applying Fatou's lemma in the same way we have 
\begin{align*}
\liminf_{t\to0^{+}}\left(\frac{\int f_{t}-\int f}{t}-\frac{\int\widetilde{f}_{t}-\int f}{t}\right) & =\liminf_{t\to0^{+}}\int\frac{f_{t}-\widetilde{f}_{t}}{t}\ge\int\left(h_{g}\left(\nabla\phi\right)f-h_{\widetilde{g}}\left(\nabla\phi\right)f\right)\\
 & =\int\left(h_{g}-h_{\widetilde{g}}\right)\dd\mu_{f},
\end{align*}
 and this time we have 
\[
\lim_{t\to0^{+}}\frac{\int\widetilde{f}_{t}-\int f}{t}=\delta(f,\widetilde{g})=\int h_{\widetilde{g}}\dd\mu_{f}+\int h_{K_{m}}\dd\nu_{f},
\]
 so we obtain 
\[
\liminf_{t\to0^{+}}\frac{\int f_{t}-\int f}{t}\ge\int h_{g}\dd\mu_{f}+\int h_{K_{m}}\dd\nu_{f}.
\]
 Since $K_{m}\subseteq K_{m+1}$ for all $m$ and $\overline{\bigcup_{m=1}^{\infty}K_{m}}=K_{g}$,
we have $h_{K_{g}}=\sup_{m}h_{K_{m}}=\lim_{m\to\infty}h_{K_{m}}$.
We may therefore let $m\to\infty$ in the last formula and deduce
that 
\[
\liminf_{t\to0^{+}}\frac{\int f_{t}-\int f}{t}\ge\int h_{g}\dd\mu_{f}+\int h_{K_{g}}\dd\nu_{f},
\]
 which together with \eqref{eq:main-upper-bound} completes the proof. 
\end{proof}
Now we can finally prove Theorem \ref{thm:main-thm}. The final step
is an approximation argument, which is essentially the same as the
one in \cite{Rotem2021}. Therefore we repeat the argument briefly
without repeating some of the computations:
\begin{proof}[Proof of Theorem \ref{thm:main-thm}]
Define a sequence $\left\{ g_{m}\right\} _{m=1}^{\infty}\subseteq\lc_{n}$
by 
\[
g_{m}(x)=\begin{cases}
g(x) & \left|x\right|\le m\\
0 & \text{otherwise. }
\end{cases}
\]
A computation shows that $h_{g_{m}}\nearrow h_{g}$ as $m\to\infty$.
Moreover, since $K_{g_{m}}\subseteq K_{g_{m+1}}$ for all $m$ and
$\overline{\bigcup_{m=1}^{\infty}K_{g_{m}}}=K_{g}$ we also have $h_{K_{g_{m}}}\nearrow h_{K_{g}}$.
If we also define $f_{t}=f\star(t\cdot g)$ and $f_{t,m}=f\star\left(t\cdot g_{m}\right)$
then another computation shows that $f_{t,m}(x)\nearrow f_{t}(x)$
for all $t>0$ and $x\in\RR^{n}$. This implies that $\int f_{t,m}\nearrow\int f_{t}$
(see e.g. Lemma 3.2 of \cite{Artstein-Avidan2004}). 

Using the chain rule for derivatives we may write 
\[
\delta(f,g)=\int f\cdot\lim_{t\to0^{+}}\frac{\log\int f_{t}-\log\int f}{t}.
\]
This formula has the advantage that by the Pr\'{e}kopa-Leindler inequality
(\cite{Prekopa1971}, \cite{Leindler1972}) the function $t\mapsto\log\int f_{t}$
is concave, so we may replace the limit by a supremum. It follows
that 
\begin{align*}
\lim_{m\to\infty}\delta(f,g_{m}) & =\sup_{m}\delta(f,g_{m})=\int f\cdot\sup_{m}\sup_{t>0}\frac{\log\int f_{t,m}-\log\int f}{t}\\
 & =\int f\cdot\sup_{t>0}\sup_{m}\frac{\log\int f_{t,m}-\log\int f}{t}=\int f\cdot\sup_{t>0}\frac{\log\int f_{t}-\log\int f}{t}=\delta(f,g).
\end{align*}
 Therefore, applying Proposition \ref{prop:main-linear-growth} and
the monotone convergence theorem we conclude that 
\[
\delta(f,g)=\lim_{m\to\infty}\delta(f,g_{m})=\lim_{m\to\infty}\left(\int h_{g_{m}}\dd\mu_{f}+\int h_{K_{g_{m}}}\dd\nu_{f}\right)=\int h_{g}\dd\mu_{f}+\int h_{K_{g}}\dd\nu_{f},
\]
 and the proof is complete. 
\end{proof}
As a corollary of the theorem we now prove that the measures $\text{\ensuremath{\mu_{f}}}$
and $\ensuremath{\nu_{f}}$ characterize the function $f$ uniquely
up to translations: 
\begin{cor}
\label{cor:min-uniq}Fix $f,g\in\lc_{n}$ with $0<\int f,\int g<\infty$
and assume that $\mu_{f}=\mu_{g}$ and $\nu_{f}=\nu_{g}$. Then there
exists $x_{0}\in\RR^{n}$ such that $f(x)=g(x-x_{0})$ .
\end{cor}

\begin{proof}
Corollary 5.3 of \cite{Colesanti2013} states that if $f,g\in\lc_{n}$
satisfy $0<\int f=\int g<\infty$, $\delta(f,g)=\delta(g,g)$ and
$\delta(g,f)=\delta(f,f)$, then there exists $x_{0}\in\RR^{n}$ such
that $f(x)=g(x-x_{0})$. This is proved by showing that we have equality
in the Pr\'{e}kopa-Leindler inequality, and using a characterization
of the equality case by Dubuc (\cite{Dubuc1977}). A similar strategy
was used in \cite{Cordero-Erausquin2015}, and indeed in the classical
proof that the surface area measure $S_{K}$ determines the body $K$
uniquely. 

In our case we have $\int f=\mu_{f}\left(\RR^{n}\right)=\mu_{g}\left(\RR^{n}\right)=\int g$,
and since 
\[
\delta(f,g)=\int_{\RR^{n}}h_{g}\dd\mu_{f}+\int_{\SS^{n-1}}h_{K_{g}}\dd\nu_{f}
\]
 we clearly have $\delta(f,g)=\delta(g,g)$ and similarly $\delta(g,f)=\delta(f,f)$.
The result follows immediately. 
\end{proof}
In \cite{Colesanti2013} the same argument was used but with Theorem
\ref{thm:colesanti-fragala} replacing Theorem \ref{thm:main-thm},
so uniqueness was only proved under the regularity assumptions of
that theorem. In \cite{Cordero-Erausquin2015} there was no explicit
representation formula for $\delta(f,g)$, but a weaker statement
that in the essentially continuous case was also sufficient in order
to reduce the uniqueness result to the equality case of Pr\'{e}kopa-Leindler
inequality (see also \cite{Rotem2021} for an explanation of why the
result of \cite{Cordero-Erausquin2015} is a weak representation theorem
for $\delta(f,g)$). We see that in order to get a clean uniqueness
result in the general case one indeed needs the full strength of Theorem
\ref{thm:main-thm}. 

Of course, Corollary \ref{cor:min-uniq} raises the question of existence:
Given measures $\mu$ and $\nu$, when is there a function $f\in\lc_{n}$
with $\mu_{f}=\mu$ and $\nu_{f}=\nu$? We believe this question can
be answered by essentially the same argument as the argument of \cite{Cordero-Erausquin2015},
which handled the case $\nu_{f}\equiv0$, but the full details are
beyond the scope of this paper. 

\bibliographystyle{plain}
\bibliography{anisotropic.bbl}

\end{document}